\documentclass[12pt]{article}
\usepackage{amsmath,amsfonts,amssymb,amsthm,txfonts,hyperref,url,t1enc}
\newtheorem{theorem}{Theorem}

\newcommand{\N}{\mathbb N}
\newcommand{\Z}{\mathbb Z}
\newcommand{\K}{\textbf{\textit{K}}}
\begin{document}
\centerline{{\Large A new characterization of computable functions}}
\vskip 0.6truecm
\noindent
\centerline{{\Large Apoloniusz Tyszka}}
\vskip 0.6truecm
\begin{sloppypar}
\noindent
{\bf Abstract.} ~~~~Let~~~~ \mbox{$E_n=\{x_i=1,~x_i+x_j=x_k,~x_i \cdot x_j=x_k: i,j,k \in \{1,\ldots,n\}\}$}.
\mbox{We prove:} {\tt (1)} there is an algorithm that for every computable function \mbox{$f:\N \to \N$}
returns a positive integer $m(f)$, for which a second algorithm accepts on the input $f$ and
any integer \mbox{$n \geq m(f)$}, and returns a system \mbox{$S \subseteq E_n$}
such that $S$ is consistent over the integers and each integer tuple $(x_1,\ldots,x_n)$ that solves
$S$ satisfies \mbox{$x_1=f(n)$}, {\tt (2)} there is an algorithm that for every computable function \mbox{$f:\N \to \N$}
returns a positive integer $w(f)$, for which a second algorithm accepts on the input $f$ and any integer
\mbox{$n \geq w(f)$}, and returns a system \mbox{$S \subseteq E_n$} such that $S$ is consistent over $\N$
and each tuple \mbox{$(x_1,\ldots,x_n)$} of non-negative integers that \mbox{solves $S$} satisfies \mbox{$x_1=f(n)$}.
\vskip 0.2truecm
\noindent
{\bf Key words and phrases:} computable function, Davis-Putnam-Robinson-Matiyasevich theorem, system of Diophantine equations.
\vskip 0.2truecm
\noindent
{\bf 2010 Mathematics Subject Classification:} 03D20, 11U99.
\vskip 0.6truecm
\par
The Davis-Putnam-Robinson-Matiyasevich theorem states that every recursively enumerable
set \mbox{${\cal M} \subseteq {\N}^n$} has a Diophantine representation, that is
\[
(a_1,\ldots,a_n) \in {\cal M} \Longleftrightarrow
\exists x_1, \ldots, x_m \in \N ~~W(a_1,\ldots,a_n,x_1,\ldots,x_m)=0
\]
for some polynomial $W$ with integer coefficients, see \cite{Matiyasevich} and \cite{Kuijer}.
The polynomial~$W$ can be computed, if we know a Turing machine~$M$
such that, for all \mbox{$(a_1,\ldots,a_n) \in {\N}^n$}, $M$ halts on \mbox{$(a_1,\ldots,a_n)$} if and only if
\mbox{$(a_1,\ldots,a_n) \in {\cal M}$}, see \cite{Matiyasevich} and \cite{Kuijer}.
\vskip 0.2truecm
\par
Before Theorem~\ref{the1}, we need an algebraic lemma together with introductory matter.
Let
\[
E_n=\{x_i=1,~x_i+x_j=x_k,~x_i \cdot x_j=x_k: i,j,k \in \{1,\ldots,n\}\}
\]
Let \mbox{$D(x_1,\ldots,x_p) \in {\Z}[x_1,\ldots,x_p]$}. For the Diophantine equation \mbox{$2 \cdot D(x_1,\ldots,x_p)=0$},
let $M$ denote the maximum of the absolute values of its coefficients.
Let ${\cal T}$ denote the family of all polynomials
$W(x_1,\ldots,x_p) \in {\Z}[x_1,\ldots,x_p]$ whose all coefficients belong to the interval $[-M,M]$
and ${\rm deg}(W,x_i) \leq d_i={\rm deg}(D,x_i)$ for each $i \in \{1,\ldots,p\}$.
Here we consider the degrees of $W(x_1,\ldots,x_p)$ and $D(x_1,\ldots,x_p)$
with respect to the variable~$x_i$. It is easy to check that
\[
{\rm card}({\cal T})=(2M+1)^{\textstyle (d_1+1) \cdot \ldots \cdot (d_p+1)}
\]
\par
We choose any bijection \mbox{$\tau: \{p+1,\ldots,{\rm card}({\cal T})\} \longrightarrow {\cal T} \setminus \{x_1,\ldots,x_p\}$}.
Let ${\cal H}$ denote the family of all equations of the forms:
\vskip 0.2truecm
\noindent
\centerline{$x_i=1$, $x_i+x_j=x_k$, $x_i \cdot x_j=x_k$~~($i,j,k \in \{1,\ldots,{\rm card}({\cal T})\})$}
\vskip 0.2truecm
\noindent
which are polynomial identities in \mbox{${\Z}[x_1,\ldots,x_p]$} if
\[
\forall s \in \{p+1,\ldots,{\rm card}({\cal T})\} ~~x_s=\tau(s)
\]
There is a unique \mbox{$q \in \{p+1,\ldots,{\rm card}({\cal T})\}$} such that \mbox{$\tau(q)=2 \cdot D(x_1,\ldots,x_p)$}.
For each ring $\K$ extending $\Z$ the system ${\cal H}$ implies \mbox{$2 \cdot D(x_1,\ldots,x_p)=x_q$}.
To see this, we observe that there exist pairwise distinct
\mbox{$t_0,\ldots,t_m \in {\cal T}$} such that $m>p$ and
\[
t_0=1~ \wedge ~t_1=x_1~ \wedge ~\ldots~ \wedge ~t_p=x_p~ \wedge ~t_m=2 \cdot D(x_1,\ldots,x_p)~ \wedge
\]
\[
\forall i \in \{p+1,\ldots,m\}~ \exists j,k \in \{0,\ldots,i-1\} ~~(t_j+t_k=t_i \vee t_i+t_k=t_j \vee t_j \cdot t_k=t_i)
\]
For each ring $\K$ extending $\Z$ and for each \mbox{$x_1,\ldots,x_p \in \K$}
there exists a unique tuple \mbox{($x_{p+1},\ldots,x_{{\rm card}({\cal T})}) \in \K^{{\rm card}({\cal T})-p}$}
such that the tuple \mbox{$(x_1,\ldots,x_p,x_{p+1},\ldots,x_{{\rm card}({\cal T})})$}
solves the system \mbox{${\cal H}$}. The sought elements \mbox{$x_{p+1},\ldots,x_{{\rm card}({\cal T})}$}
are given by the formula
\[
\forall s \in \{p+1,\ldots,{\rm card}({\cal T})\} ~~x_s=\tau(s)(x_1,\ldots,x_p)
\]
This proves the following Lemma.
\vskip 0.2truecm
\noindent
{\bf Lemma.} {\em The system ${\cal H} \cup \{x_q+x_q=x_q\}$ can be simply computed.
For each ring $\K$ extending $\Z$, the equation $D(x_1,\ldots,x_p)=0$
is equivalent to the system ${\cal H} \cup \{x_q+x_q=x_q\} \subseteq E_{{\rm card}({\cal T})}$.
Formally, this equivalence can be written as
\[
\forall x_1,\ldots,x_p \in \K ~\Bigl(D(x_1,\ldots,x_p)=0 \Longleftrightarrow
\exists x_{p+1},\ldots,x_{{\rm card}({\cal T})} \in \K
\]
\[
(x_1,\ldots,x_p,x_{p+1},\ldots,x_{{\rm card}({\cal T})}) {\rm ~solves~the~system~}
{\cal H} \cup \{x_q+x_q=x_q\} \Bigr)
\]
For each ring $\K$ extending $\Z$ and for each \mbox{$x_1,\ldots,x_p \in \K$} with
\mbox{$D(x_1,\ldots,x_p)=0$} there exists a unique tuple
\mbox{($x_{p+1},\ldots,x_{{\rm card}({\cal T})}) \in \K^{{\rm card}({\cal T})-p}$} such
that the tuple \mbox{$(x_1,\ldots,x_p,x_{p+1},\ldots,x_{{\rm card}({\cal T})})$} solves the system
\mbox{${\cal H} \cup \{x_q+x_q=x_q\}$}. Hence, for each ring $\K$ extending $\Z$ the equation
\mbox{$D(x_1,\ldots,x_p)=0$} has the same number of solutions as the system \mbox{${\cal H} \cup \{x_q+x_q=x_q\}$}.}
\vskip 0.2truecm
\par
Putting $M=M/2$ we obtain new families ${\cal T}$ and ${\cal H}$.
There is a unique \mbox{$q \in \{1,\ldots,{\rm card}({\cal T})\}$} such that
\[
\Bigl(q \in \{1,\ldots,p\}~ \wedge ~x_q=D(x_1,\ldots,x_p)\Bigr)~ \vee
\]
\[
\Bigl(q \in \{p+1,\ldots,{\rm card}({\cal T})\}~ \wedge ~\tau(q)=D(x_1,\ldots,x_p)\Bigr)
\]
The new system \mbox{${\cal H} \cup \{x_q+x_q=x_q\}$} is equivalent to \mbox{$D(x_1,\ldots,x_p)=0$}
and can be simply computed.
\begin{theorem}\label{the1}
There is an algorithm that for every computable
function \mbox{$f:\N \to \N$} returns a positive integer $m(f)$, for which a second algorithm accepts on the
input $f$ and any integer \mbox{$n \geq m(f)$}, and returns a system \mbox{$S \subseteq E_n$} such that $S$
is consistent over the integers and each integer tuple \mbox{$(x_1,\ldots,x_n)$} that solves $S$ satisfies \mbox{$x_1=f(n)$}.
\end{theorem}
\begin{proof}
By the Davis-Putnam-Robinson-Matiyasevich theorem, the function $f$ has a Diophantine representation.
It means that there is a polynomial $W(x_1,x_2,x_3,\ldots,x_r)$ with integer coefficients
such that for each non-negative integers $x_1$, $x_2$,
\begin{equation}
\tag*{\tt (E1)}
x_1=f(x_2) \Longleftrightarrow \exists x_3, \ldots, x_r \in \N ~~W(x_1,x_2,x_3,\ldots,x_r)=0
\end{equation}
By the equivalence~{\rm (E1)} and Lagrange's four-square theorem, for each integers $x_1$, $x_2$,
the conjunction \mbox{$(x_2 \geq 0) \wedge (x_1=f(x_2))$} holds true if and only if there exist integers
\[
a,b,c,d,\alpha,\beta,\gamma,\delta,x_3,x_{3,1},x_{3,2},x_{3,3},x_{3,4},\ldots,x_r,x_{r,1},x_{r,2},x_{r,3},x_{r,4}
\]
such that
\[
W^2(x_1,x_2,x_3,\ldots,x_r)+\bigl(x_1-a^2-b^2-c^2-d^2\bigr)^2+\bigl(x_2-\alpha^2-\beta^2-\gamma^2-\delta^2\bigr)^2+
\]
\[
\bigl(x_3-x^2_{3,1}-x^2_{3,2}-x^2_{3,3}-x^2_{3,4}\bigr)^2+\ldots+\bigl(x_r-x^2_{r,1}-x^2_{r,2}-x^2_{r,3}-x^2_{r,4}\bigr)^2=0
\]
\newpage
\noindent
By the Lemma, there is an integer \mbox{$s \geq 3$} such that for each integers $x_1$, $x_2$,
\begin{equation}
\tag*{\tt (E2)}
\Bigl(x_2 \geq 0 \wedge x_1=f(x_2)\Bigr) \Longleftrightarrow \exists x_3,\ldots,x_s \in \Z ~~\Psi(x_1,x_2,x_3,\ldots,x_s)
\end{equation}
where the formula $\Psi(x_1,x_2,x_3,\ldots,x_s)$ is algorithmically determined as a conjunction of formulae of the forms:
\vskip 0.2truecm
\centerline{$x_i=1,~~x_i+x_j=x_k,~~x_i \cdot x_j=x_k~~(i,j,k \in \{1,\ldots,s\})$}
\vskip 0.1truecm
\noindent
Let $m(f)=4+2s$, and let $[\cdot]$ denote the integer part function. For each integer $n \geq m(f)$,
\[
n-\left[\frac{n}{2}\right]-2-s \geq m(f)-\left[\frac{m(f)}{2}\right]-2-s \geq m(f)-\frac{m(f)}{2}-2-s=0
\]
Let $S$ denote the following system
\[\left\{
\begin{array}{rcl}
{\rm all~equations~occurring~in~}\Psi(x_1,x_2,x_3,\ldots,x_s) \\
n-\left[\frac{n}{2}\right]-2-s {\rm ~equations~of~the~form~} z_i=1 \\
t_1 &=& 1 \\
t_1+t_1 &=& t_2 \\
t_2+t_1 &=& t_3 \\
&\ldots& \\
t_{\left[\frac{n}{2}\right]-1}+t_1 &=& t_{\left[\frac{n}{2}\right]} \\
t_{\left[\frac{n}{2}\right]}+t_{\left[\frac{n}{2}\right]} &=& w \\
w+y &=& x_2 \\
y+y &=& y {\rm ~(if~}n{\rm ~is~even)} \\
y &=& 1 {\rm ~(if~}n{\rm ~is~odd)}
\end{array}
\right.\]
with $n$ variables. By the equivalence~{\tt (E2)}, the system~$S$ is consistent over $\Z$.
If an integer $n$-tuple $(x_1,x_2,x_3,\ldots,x_s,\ldots,w,y)$ solves~$S$,
then by the equivalence~{\tt (E2)},
\[
x_1=f(x_2)=f(w+y)=f\left(2 \cdot \left[\frac{n}{2}\right]+y\right)=f(n)
\]
\end{proof}
\par
A simpler proof, not using Lagrange's four-square theorem, suffices if we consider solutions in non-negative integers.
\begin{theorem}\label{the2}
There is an algorithm that for every computable function \mbox{$f:\N \to \N$} returns
a positive integer $w(f)$, for which a second algorithm accepts on the input $f$ and any integer \mbox{$n \geq w(f)$},
and returns a system \mbox{$S \subseteq E_n$} such that $S$ is consistent over $\N$ and each tuple \mbox{$(x_1,\ldots,x_n)$}
of non-negative integers that solves $S$ satisfies \mbox{$x_1=f(n)$}.
\end{theorem}
\begin{proof}
We omit the construction of $S$ because a similar construction is carried out in the proof of Theorem~\ref{the1}.
As we now consider solutions in $\N$, we need a new algorithm
which transforms any Diophantine equation into an equivalent system of equations of the forms:
\vskip 0.1truecm
\centerline{$x_i=1,~~x_i+x_j=x_k,~~x_i \cdot x_j=x_k$}
\vskip 0.1truecm
\noindent
Let \mbox{$D(x_1,\ldots,x_p) \in {\Z}[x_1,\ldots,x_p] \setminus \{0\}$}, and let
\[
D(x_1,\ldots,x_p)=\sum a(i_1,\ldots,i_p) \cdot x_1^{\textstyle i_1} \cdot \ldots \cdot x_p^{\textstyle i_p}
\]
\vskip 0.1truecm
\noindent
where \mbox{$a(i_1,\ldots,i_p)$} denote non-zero integers. Let
\[
B(x_1,\ldots,x_p)=\sum (|a(i_1,\ldots,i_p)|+2) \cdot x_1^{\textstyle i_1} \cdot \ldots \cdot x_p^{\textstyle i_p}
\]
\par
\[
A(x_1,\ldots,x_p)=D(x_1,\ldots.x_p)+B(x_1,\ldots,x_p)
\]
\vskip 0.4truecm
\noindent
Then, the equation \mbox{$D(x_1,\ldots,x_p)=0$} is equivalent to \mbox{$A(x_1,\ldots,x_p)=$} \mbox{$B(x_1,\ldots,x_p)$}.
The polynomials \mbox{$A(x_1,\ldots,x_p)$} and \mbox{$B(x_1,\ldots,x_p)$} have positive integer coefficients and
\[
A(x_1,\ldots,x_p) \not\in\ \{x_1,\ldots,x_p,0\} \wedge B(x_1,\ldots,x_p) \not\in \{x_1,\ldots,x_p,0, A(x_1,\ldots,x_p)\}
\]
Let $\delta$ denote the maximum of the coefficients
of \mbox{$A(x_1,\ldots,x_p)$} and \mbox{$B(x_1,\ldots,x_p)$}, and let ${\cal T}$ denote the family of all polynomials
$W(x_1,\ldots,x_p) \in {\Z}[x_1,\ldots,x_p]$ whose coefficients belong to the interval $[0,~\delta]$ and
\[
{\rm deg}(W,x_i) \leq {\rm max}\Bigl({\rm deg}(A,x_i),~{\rm deg}(B,x_i)\Bigr)
\]
for each \mbox{ $i \in \{1,\ldots,p\}$}. Here we consider the degrees with respect to the variable~$x_i$.
Let $n$ denote the cardinality of~${\cal T}$. We choose any bijection
\[
\tau: \{p+1,\ldots,n\} \longrightarrow {\cal T} \setminus \{x_1,\ldots,x_p\}
\]
such that $\tau(p+1)=0$, $\tau(p+2)=A(x_1,\ldots,x_p)$, and $\tau(p+3)=B(x_1,\ldots,x_p)$.
Let ${\cal H}$ denote the family of all equations of the form
\[
x_i=1,~~x_i+x_j=x_k,~~x_i \cdot x_j=x_k~~(i,j,k \in \{1,\ldots,n\})
\]
which are polynomial identities in \mbox{${\Z}[x_1,\ldots,x_p]$} if
\[
\forall s \in \{p+1,\ldots,n\} ~~x_s=\tau(s)
\]
Since $\tau(p+1)=0$, the equation $x_{p+1}+x_{p+1}=x_{p+1}$ belongs to ${\cal H}$.
Let
\[
T={\cal H} \cup \{x_{p+1}+x_{p+2}=x_{p+3}\}
\]
The system $T$ can be computed, \mbox{$T \subseteq E_n$}, and
\[
\forall x_1,\ldots,x_p \in \N ~\Bigl(D(x_1,\ldots,x_p)=0 \Longleftrightarrow
\]
\[
\exists x_{p+1},\ldots,x_n \in \N ~(x_1,\ldots,x_p,x_{p+1},\ldots,x_n) {\rm ~solves~} T\Bigr)
\]
For each \mbox{$x_1,\ldots,x_p \in \N$} with
\mbox{$D(x_1,\ldots,x_p)=0$} there exists a unique tuple \mbox{$(x_{p+1},\ldots,x_n) \in {\N}^{n-p}$}
such that the tuple \mbox{$(x_1,\ldots,x_p,x_{p+1},\ldots,x_n)$} solves $T$.
Hence, the equation \mbox{$D(x_1,\ldots,x_p)=0$} has the same number of non-negative integer
solutions as $T$.
\end{proof}
\end{sloppypar}

\noindent
Apoloniusz Tyszka\\
Technical Faculty\\
Hugo Ko\l{}\l{}\k{a}taj University\\
Balicka 116B, 30-149 Krak\'ow, Poland\\
E-mail address: \url{rttyszka@cyf-kr.edu.pl}
\end{document}